\documentclass[12pt,reqno]{amsart}

\usepackage{amssymb}
\usepackage[all]{xy}

\usepackage{hyperref}

\numberwithin{equation}{section}

\newtheorem{theorem}{Theorem}[section]
\newtheorem{proposition}[theorem]{Proposition}
\newtheorem{lemma}[theorem]{Lemma}
\newtheorem{corollary}[theorem]{Corollary}

\theoremstyle{definition}
\newtheorem{definition}[theorem]{Definition}
\newtheorem{remark}[theorem]{Remark}
\newtheorem{example}[theorem]{Example}

\begin{document}

\baselineskip=15pt

\title[Lie algebroid connections on holomorphic principal bundles]{Holomorphic Lie algebroid
connections on holomorphic principal bundles on compact Riemann surfaces}

\author[I. Biswas]{Indranil Biswas}

\address{Department of Mathematics, Shiv Nadar University, NH91, Tehsil Dadri,
Greater Noida, Uttar Pradesh 201314, India}

\email{indranil.biswas@snu.edu.in, indranil29@gmail.com}



\date{}

\begin{abstract}
For a $\Gamma$--equivariant holomorphic Lie algebroid $(V,\, \phi)$, on a compact Riemann surface $X$ equipped
with an action of a finite group $\Gamma$, we investigate the equivariant holomorphic Lie algebroid
connections on holomorphic principal $G$--bundles over $X$, where $G$ is a connected affine complex reductive
group. If $(V,\,\phi)$ is nonsplit, then it is proved that every holomorphic principal $G$--bundle admits
an equivariant holomorphic Lie algebroid connection. If $(V,\,\phi)$ is split, then it is proved that the following
four statements are equivalent:
\begin{enumerate}
\item An equivariant principal $G$--bundle $E_G$ admits an equivariant holomorphic Lie algebroid connection.

\item The equivariant principal $G$--bundle $E_G$ admits an equivariant holomorphic connection.

\item The principal $G$--bundle $E_G$ admits a holomorphic connection.

\item For every triple $(P,\, L(P),\, \chi)$, where $L(P)$ is a Levi subgroup of a parabolic subgroup
$P\, \subset\, G$ and $\chi$ is a holomorphic character of $L(P)$,
and every $\Gamma$--equivariant holomorphic reduction of structure group $E_{L(P)}$ of $E_G$ to $L(P)$,
the degree of the line bundle over $X$ associated to $E_{L(P)}$ for $\chi$ is zero.
\end{enumerate}

The correspondence between $\Gamma$--equivariant principal $G$--bundles over $X$ and parabolic
$G$--bundles on $X/\Gamma$ translates the above result to the context of parabolic $G$--bundles.
\end{abstract} 

\maketitle

\tableofcontents

\section{Introduction}

A well-known theorem of Atiyah and Weil says the following: A holomorphic vector bundle $E$ on a compact
connected Riemann surface $X$ admits a holomorphic connection if and only if the degree of each indecomposable
component of $E$ is zero \cite{At}, \cite{We}. This criterion for the existence of holomorphic connections
extends to holomorphic principal $G$--bundles over $X$, where $G$ is a reductive affine algebraic group
defined over $\mathbb C$ \cite{AB}.

The notion of a holomorphic connection on a holomorphic vector bundle $E$ extends to the notion of holomorphic
Lie algebroid connections on $E$, which we briefly recall.

A holomorphic Lie algebroid over $X$ is a pair $(V,\,\phi)$, where $V$ is a holomorphic vector bundle over $X$
equipped with the structure of a $\mathbb{C}$--bilinear Lie algebra on its sheaf of holomorphic sections, and
$\phi\,:\, V\, \longrightarrow\, TX$ is an ${\mathcal O}_X$--linear homomorphism satisfying the Leibniz rule
$$
[s,\, f t] \ =\ f\,[s,\, t]\;+\;\phi(s)(f)\,t
$$
for all locally defined holomorphic sections $s,\,t$ of $V$ and all locally defined holomorphic
functions $f$.

We work with an equivariant set-up, meaning a finite subgroup $\Gamma\, \subset\, \text{Aut}(X)$
is fixed, and all objects and structures on $X$ are taken to be $\Gamma$--equivariant.

A Lie algebroid $(V,\,\phi)$ is called split if there is a holomorphic $\Gamma$–equivariant homomorphism
$\eta\, :\, TX\,\longrightarrow\, V$ such that $\phi\circ\eta\,=\, {\rm Id}_{TX}$. If $(V,\, \phi)$ is
not split, then it is called nonsplit. See Example \ref{ex1} for nonsplit and split Lie algebroids.

Let $E_G$ be a $\Gamma$--equivariant holomorphic principal $G$--bundle over $X$, where $G$, as before,
is a reductive affine algebraic group defined over $\mathbb C$. Using the Atiyah bundle for $E_G$ and
the pair $(V,\, \phi)$, a $\mathbb C$--Lie algebra bundle ${\mathcal A}(E_G)$ is constructed which fits
in the following short exact sequence of $\Gamma$--equivariant vector bundles over $X$:
$$
0\ \longrightarrow\ \mathrm{ad}(E_G)\ \longrightarrow\ \mathcal{A}(E_G)\ \xrightarrow{\,\,\,\rho
\,\,\,}\ V\ \longrightarrow\ 0,
$$
where $\mathrm{ad}(E_G)$ is the adjoint bundle for $E_G$. An equivariant holomorphic Lie algebroid connection
on $E_G$ is a $\Gamma$--equivariant holomorphic homomorphism $\delta
\, :\, V\, \longrightarrow\, \mathcal{A}(E_G)$ such that
$\rho\circ\delta\,=\,{\rm Id}_V$. In the special case where $(V,\, \phi)\,=\, (TX,\, {\rm Id}_{TX})$, a
holomorphic Lie algebroid connection on $E_G$ is a usual holomorphic connection on $E_G$.

There is a large body of research on Lie algebroids and Lie algebroid connections.
Bruzzo and Rubtsov investigated the cohomology and moduli spaces of skew-holomorphic 
Lie algebroids \cite{BR}. Tortella introduced modules over Lie algebroids and described moduli space
of flat Lie algebroid connections which are also called $\Lambda$–modules \cite{To1}, \cite{To2}. In \cite{AO},
Alfaya and Oliveira studied the moduli space of flat Lie algebroid connections and proved numerous properties
of the moduli space \cite{AO}. Bruzzo-Mencattini-Rubtsov-Tortella investigated extensions of Lie algebroids
\cite{BMRT}. Laurent-Gengoux, Sti\'enon and Xu investigate the relationships between holomorphic Lie
algebroids and holomorphic Poisson structures.

Our aim here is to give a criterion for the existence of equivariant holomorphic Lie algebroid connections on
an equivariant holomorphic principal $G$--bundle over $X$. We prove the following (see Theorem \ref{thm1}):

\begin{theorem}\label{thm-i}
\mbox{}
\begin{itemize}
\item Let $(V,\,\phi)$ be a nonsplit $\Gamma$--equivariant Lie algebroid. Then any
equivariant principal $G$--bundle over $X$ admits an equivariant holomorphic Lie algebroid connection.

\item Let $(V,\, \phi)$ be a split $\Gamma$--equivariant Lie algebroid.
Let $E_G$ be an equivariant principal $G$--bundle over $X$. The following four statements
are equivalent:
\begin{enumerate}
\item $E_G$ over $X$ admits an equivariant holomorphic Lie algebroid connection.

\item $E_G$ admits an equivariant holomorphic connection.

\item $E_G$ admits a holomorphic connection.

\item For every triple $(P,\, L(P),\, \chi)$, where $L(P)$ is a Levi subgroup of a parabolic subgroup
$P\, \subset\, G$ and $\chi$ is a holomorphic character of $L(P)$,
and every $\Gamma$--equivariant holomorphic reduction of structure group $E_{L(P)}$ of $E_G$ to $L(P)$,
the degree of the line bundle over $X$ associated to $E_{L(P)}$ for $\chi$ is zero.
\end{enumerate}
\end{itemize}
\end{theorem}

In \cite{ABKS}, Theorem \ref{thm-i} was proved under the assumption that $G\,=\, \text{GL}(r,{\mathbb C})$.

There is a natural bijective correspondence between parabolic $G$--bundles on $X/\Gamma$ and $\Gamma$--equivariant
principal $G$--bundles over $X$. Using this correspondence, Theorem \ref{thm-i} translates into the following
(see Theorem \ref{thm2}):

\begin{theorem}\label{thm-i2}
Let $Y$ be a compact connected Riemann surface and $\{s_1,\, \cdots,\, s_n\}\, \subset\, Y$ a
parabolic divisor. Fix an integer $N_i\, \geq\, 2$ for each $s_i$, $1\, \leq\, i\, \leq\, n$.
\begin{itemize}
\item Let $(V_*,\,\phi)$ be a nonsplit parabolic Lie algebroid on $Y$. Then any
parabolic $G$--bundle on $Y$ admits a parabolic Lie algebroid connection.

\item Let $(V_*,\, \phi)$ be a split parabolic Lie algebroid on $Y$.
Let ${\mathcal E}_G$ be a parabolic $G$--bundle on $Y$. The following three statements
are equivalent:
\begin{enumerate}
\item ${\mathcal E}_G$ admits a parabolic Lie algebroid connection.

\item ${\mathcal E}_G$ admits a parabolic holomorphic connection.

\item For every triple $(P,\, L(P),\, \chi)$, where $L(P)$ is a Levi subgroup of a parabolic subgroup
$P\, \subset\, G$ and $\chi$ is a holomorphic character of $L(P)$, and every
holomorphic reduction of structure group ${\mathcal E}_{L(P)}$ of ${\mathcal E}_G$ to $L(P)$, the
parabolic line bundle over $X$ associated to ${\mathcal E}_{L(P)}$ for $\chi$ has parabolic degree zero.
\end{enumerate}
\end{itemize}
\end{theorem}

\section{Equivariant Lie algebroids}\label{sec:Equivariant Lie algebroids}

Let $X$ be a compact connected Riemann surface. Denote by $\text{Aut}(X)$ the group of all
holomorphic automorphisms of $X$. Fix a finite subgroup
\begin{equation}\label{e1}
\Gamma \ \subset\ \text{Aut}(X).
\end{equation}
So the group $\Gamma$ has a tautological action on $X$.

Let $\mathbb G$ be a complex Lie group. Note that in the introduction $G$ was
a reductive affine algebraic group defined over $\mathbb C$.
An \textit{equivariant} principal ${\mathbb G}$--bundle
over $X$ is a holomorphic principal $\mathbb G$--bundle
\begin{equation}\label{e2}
p\ :\ E_{\mathbb G}\ \longrightarrow\ X
\end{equation}
over $X$ together with an action of $\Gamma$ on $E_{\mathbb G}$ such that
\begin{enumerate}
\item for every $\gamma\, \in\, \Gamma$, the automorphism of $E_{\mathbb G}$ given by the action of
$\gamma$ is holomorphic,

\item the projection $p$ in \eqref{e2} is $\Gamma$--equivariant, and

\item the actions of $\mathbb G$ and $\Gamma$ on $E_{\mathbb G}$ commute.
\end{enumerate}
A holomorphic vector bundle $V$ of rank $r$ over $X$ is called equivariant if the corresponding holomorphic
principal ${\rm GL}(r, {\mathbb C})$--bundle over $X$, given by the frames in the fibers of $V$,
is equipped with an action of $\Gamma$ that satisfies the above three conditions. This is equivalent to
an action of $\Gamma$ on $V$, via holomorphic vector bundle automorphisms, over the action of $\Gamma$
on $X$.

The holomorphic tangent bundle
of $X$ will be denoted by $TX$, while the holomorphic cotangent bundle
of $X$ will be denoted by $K_X$. Using the action of $\Gamma$ on $X$, both $TX$ and $K_X$
are equivariant line bundles.

The first jet bundle of a holomorphic vector bundle
$W$ over $X$ will be denoted by $J^1(W)$.
An equivariant $\mathbb{C}$--Lie algebra structure on an equivariant vector
bundle $V$ over $X$ is a $\mathbb{C}$--bilinear pairing defined by a sheaf homomorphism
$$
[-,\, -] \,\,:\,\, V\otimes_{\mathbb C} V \,\, \longrightarrow\,\, V,
$$
which is given by a holomorphic homomorphism $J^1(V)\otimes J^1(V)\, \longrightarrow\, V$
of vector bundles, such that
\begin{enumerate}
\item $[\gamma(s),\, \gamma(t)]\,=\, \gamma([s,\, t])$ for all $\gamma\, \in\, \Gamma$, and

\item $[s,\, t]\,=\, -[t,\, s]$ and $[[s,\, t],\, u]+[[t,\, u],\, s]+[[u,\, s],\, t]\,=\,0$
for all locally defined holomorphic sections $s,\, t,\, u$ of $V$.
\end{enumerate}
The Lie bracket operation on the sheaf of holomorphic vector fields on $X$ gives the structure of
an equivariant $\mathbb C$--Lie algebra on $TX$.

An equivariant Lie algebroid over $X$ is a pair $(V,\, \phi)$, where
\begin{enumerate}
\item $V$ is an equivariant vector bundle over $X$ equipped with the structure of an
equivariant $\mathbb{C}$--Lie algebra, and

\item $\phi\, :\, V\, \longrightarrow\, TX$ is a $\Gamma$--equivariant ${\mathcal O}_X$--linear homomorphism
such that
\begin{equation}\label{esc}
[s,\, f\cdot t]\ =\ f\cdot [s,\, t]+\phi(s)(f)\cdot t
\end{equation}
for all locally defined holomorphic sections
$s,\, t$ of $V$ and all locally defined holomorphic functions $f$ on $X$.
\end{enumerate}
The above homomorphism $\phi$ is called the \textit{anchor map} of the Lie algebroid.
The two conditions in the definition of a Lie algebroid imply that
\begin{equation}\label{e5}
\phi([s,\, t])\ =\ [\phi(s),\, \phi(t)]
\end{equation}
for all locally defined holomorphic sections $s,\, t$ of $V$; this is explained in Remark
\ref{rem-e} below.

\begin{remark}\label{rem-e}
To show that \eqref{e5} holds for $(V,\, \phi)$, note that for
all holomorphic local sections $s,\,t,\,u$ of $V$ and each locally defined
holomorphic function $f$ in $\mathcal{O}_X$ we have
\begin{equation}\label{b1}
[[s,\,t],\,fu]\ =\ f[[s,\,t],\,u]+\phi([s,\,t])(f)\cdot u
\end{equation}
(see \eqref{esc}). On the other hand,
$$
[[s,\,t],\,fu]\,=\,[[s,\,fu],\,t]+[s,\,[t,\,fu]]\ =\ [f[s,\,u]+\phi(s)(f)u,\,t]
$$
$$
+[s,\,f[t,\,u]+\phi(t)(f) u] \ =\ f[[s,\,u],\,t]-\phi(t)(f)[s,\,u]+\phi(s)(f)[u,\,t]
$$
$$
-\phi(t)(\phi(s)(f)) u + f[s,\,[t,\,u]]+\phi(s)(f)[t,\,u] +\phi(t)(f)[s,\,u]+\phi(s)(\phi(t)(f))u
$$
$$
=\ f[[s,\,t],\,u]+\left( \phi(s)(\phi(t)(f))-\phi(t)(\phi(s)(f))\right) u
$$
\begin{equation}\label{b2}
=\, f[[s,\,t],\,u]+[\phi(s),\, \phi(t)](f)\cdot u.
\end{equation}
Combining \eqref{b1} and \eqref{b2} we conclude that $\phi([s,\,t])(f)\cdot  u\,=\, [\phi(s),\,
\phi(t)](f)\cdot u$. Since this holds for all locally defined $f$ and $u$, it follows that
$$\phi([s,\,t])\,\ =\,\ [\phi(s),\,\phi(t)].$$
This proves \eqref{e5}.
\end{remark}

\begin{definition}\label{def1}
An equivariant Lie algebroid $(V,\, \phi)$ over $X$ will be called \textit{split} if there is
a $\Gamma$--equivariant ${\mathcal O}_X$--linear homomorphism
$$
\rho\ :\ TX \ \longrightarrow\ V
$$
such that $\phi\circ\rho\,=\, {\rm Id}_{TX}$. An equivariant Lie algebroid $(V,\, \phi)$ over $X$ will
be called \textit{nonsplit} if it is not split.
\end{definition}

\begin{lemma}\label{lem0}
Let $(V,\, \phi)$ be an equivariant Lie algebroid over $X$.
If there is an ${\mathcal O}_X$--linear homomorphism
$$
\zeta\ :\ TX \ \longrightarrow\ V
$$
such that $\phi\circ\zeta\,=\, {\rm Id}_{TX}$, then there is
a $\Gamma$--equivariant ${\mathcal O}_X$--linear homomorphism
$$
\widehat{\zeta}\ :\ TX \ \longrightarrow\ V
$$
such that $\phi\circ\widehat{\zeta}\,=\, {\rm Id}_{TX}$.
\end{lemma}

\begin{proof}
Let $\zeta\, :\, TX \, \longrightarrow\, V$ be an ${\mathcal O}_X$--linear homomorphism
such that $\phi\circ\zeta\,=\, {\rm Id}_{TX}$. For each $\gamma\,\in\, \Gamma$, let
$$
\zeta_\gamma\ :\ TX \ \longrightarrow\ V
$$
be the homomorphism given by the following composition of maps:
$$
TX \ \xrightarrow{\,\,\, \gamma\cdot\,\,\,}\ TX\ \xrightarrow{\,\,\,\zeta\,\,\,}\ V
\ \xrightarrow{\,\,\, \gamma^{-1}\cdot\,\,\,}\ V,
$$
where $\gamma\cdot$ (respectively, $\gamma^{-1}\cdot$) is the action of $\gamma$
(respectively, $\gamma^{-1}$) on $TX$ (respectively, $V$). Then the homomorphism
$$
\widehat{\zeta}\, :=\ \frac{1}{\#\Gamma} \sum_{\gamma\in \Gamma} \zeta_\gamma
\ :\ TX \ \longrightarrow\ V,
$$
where $\# \Gamma$ is the order of $\Gamma$, is clearly $\Gamma$--equivariant and it also satisfies
the condition that $\phi\circ\widehat{\zeta}\,=\, {\rm Id}_{TX}$.
\end{proof}

\section{Lie algebroid connection on principal bundles}\label{sec:Lie algebroid connection on principal bundles}

As before, $\mathbb G$ is a complex Lie group. Take an
equivariant principal $\mathbb G$--bundle $p\, :\, E_{\mathbb G}\, \longrightarrow\, X$ (see \eqref{e2}). The
action of $\mathbb G$
on $E_{\mathbb G}$ produces an action of $\mathbb G$ on the direct image $p_*TE_{\mathbb G}$ of
the holomorphic tangent bundle $TE_{\mathbb G}$. The quotient
\begin{equation}\label{e6}
\psi\ :\ \text{At}(E_{\mathbb G})\ :=\ (p_*TE_{\mathbb G})/{\mathbb G} \ \longrightarrow\ X
\end{equation}
is the Atiyah bundle for $E_{\mathbb G}$ \cite{At}. It fits in a short exact sequence of holomorphic vector bundles
\begin{equation}\label{e7}
0\, \longrightarrow\, \text{ad}(E_{\mathbb G})\, \xrightarrow{\,\,\,\iota\,\,\,} \, \text{At}(E_{\mathbb G})
\, \xrightarrow{\,\,\,\varpi\,\,\,} \, TX\, \longrightarrow\, 0,
\end{equation}
where $\text{ad}(E_{\mathbb G})$ is the adjoint vector bundle for $E_{\mathbb G}$ (see \cite{At}); the projection
$\varpi$ in \eqref{e7} is given by the differential $dp\, :\, TE_{\mathbb G}\, \longrightarrow\, p^*TX$
of the map $p$. The sequence in \eqref{e7} is known as the Atiyah exact sequence for $E_{\mathbb G}$.

The Lie bracket operation on the sheaf of holomorphic vector fields on
$E_{\mathbb G}$ produces
a $\mathbb C$--Lie algebra structure on $\text{At}(E_{\mathbb G})$. The homomorphism $\varpi$ in \eqref{e7}
intertwines the $\mathbb C$--Lie algebra structures of $\text{At}(E_{\mathbb G})$ and $TX$. In fact,
$(\text{At}(E_{\mathbb G}), \,\varpi)$ is a Lie algebroid.

A holomorphic connection on the principal $\mathbb G$--bundle $E_{\mathbb G}$ is a holomorphic splitting
of the Atiyah exact sequence in \eqref{e7} \cite{At}. In other words, a holomorphic connection on
$E_{\mathbb G}$ is a holomorphic ${\mathcal O}_X$--linear homomorphism $\mu\, :\, TX\, \longrightarrow\,
\text{At}(E_{\mathbb G})$ such that $\varpi\circ\mu \,=\, \text{Id}_{TX}$, where $\varpi$ is the
homomorphism in \eqref{e7}.

\begin{example}\label{ex1}
Assume that $E_{\mathbb G}$ does not admit any holomorphic connection. For example, set ${\mathbb G}\,=\,
\text{GL}(r,{\mathbb C})$ and take $E_{\mathbb G}$ to be the holomorphic principal
$\text{GL}(r,{\mathbb C})$--bundle over $X$ associated to a holomorphic vector bundle of rank
$r$ and nonzero degree over $X$. Then the Lie algebroid $(\text{At}(E_{\mathbb G}),\, \varpi)$
in \eqref{e7} is nonsplit.

On the other hand, if $E_{\mathbb G}$ admits a holomorphic connection, then
the Lie algebroid $(\text{At}(E_{\mathbb G}),\, \varpi)$
in \eqref{e7} is split. For example, take any indecomposable holomorphic vector bundle $E$ over $X$
of rank $r$ with $\text{degree}(E)\,=\, 0$. Then $E$ admits a holomorphic connection \cite{At}, \cite{We}.
Hence the holomorphic principal $\text{GL}(r,{\mathbb C})$--bundle $E_{\text{GL}(r,{\mathbb C})}$ over $X$
associated to $E$ admits a holomorphic connection. Consequently, the Lie algebroid
given by $\text{At}(E_{\text{GL}(r,{\mathbb C})})$ (see \eqref{e7}) is split.
\end{example}

The action of $\Gamma$ on $E_{\mathbb G}$ makes $\text{At}(E_{\mathbb G})$ an equivariant vector bundle. The 
homomorphism $\varpi$ in \eqref{e7} is $\Gamma$--equivariant. Thus $(\text{At}(E_{\mathbb G}), \,\varpi)$ is
an equivariant Lie algebroid. The action of $\Gamma$ on $E_{\mathbb G}$ produces an action of $\Gamma$
on $\text{ad}(E_{\mathbb G})$, and the homomorphism $\iota$ in \eqref{e7} is $\Gamma$--equivariant.

Take an equivariant Lie algebroid $(V,\, \phi)$ over $X$. Consider the homomorphism
$$
\psi\ :\ V\oplus \text{At}(E_{\mathbb G}) \ \longrightarrow\ TX, \ \ \, (v,\, w)\
\longmapsto\ \phi(v) - \varpi(w),
$$
where $\varpi$ is the homomorphism in \eqref{e7}. Note that $\psi$ is surjective because $\varpi$
is surjective. Define
\begin{equation}\label{e8}
\mathcal{A}(E_{\mathbb G}) \ :=\ \text{kernel}(\psi)\ \subset\ V\oplus \text{At}(E_{\mathbb G}).
\end{equation}
The $\mathbb C$--Lie algebra structure on $V\oplus \text{At}(E_{\mathbb G})$, given by the
$\mathbb C$--Lie algebra structures on $V$ and $\text{At}(E_{\mathbb G})$, restricts to a
$\mathbb C$--Lie algebra structure on $\mathcal{A}(E_{\mathbb G})$.
Restricting the natural projection $V\oplus \text{At}(E_{\mathbb G})\, \longrightarrow\, V$ to
$\mathcal{A}(E_{\mathbb G})\, \subset\, V\oplus \text{At}(E_{\mathbb G})$ we obtain a homomorphism
\begin{equation}\label{e9}
\rho\ :\ \mathcal{A}(E_{\mathbb G}) \ \longrightarrow\ V;
\end{equation}
note that $\text{kernel}(\rho) \,=\, \text{kernel}(\varpi)\,=\, \text{ad}(E_{\mathbb G})$.
Similarly, restricting the natural projection $V\oplus \text{At}(E_{\mathbb G})\, \longrightarrow\,
\text{At}(E_{\mathbb G})$
to $\mathcal{A}(E_{\mathbb G})\, \subset\, V\oplus \text{At}(E_{\mathbb G})$ we obtain a homomorphism
\begin{equation}\label{e10}
\varphi\ :\ \mathcal{A}(E_{\mathbb G}) \ \longrightarrow\ \text{At}(E_{\mathbb G}).
\end{equation}
The action of $\Gamma$ on $V\oplus \text{At}(E_{\mathbb G})$, given by the actions of $\Gamma$ on $V$ and
$\text{At}(E_{\mathbb G})$, preserves the subbundle $\mathcal{A}(E_{\mathbb G})$.

We have the commutative
diagram of homomorphisms of vector bundles
\begin{equation}\label{e11}
\begin{matrix}
0 &\longrightarrow & {\rm ad}(E_{\mathbb G}) & \longrightarrow & {\mathcal A}(E_{\mathbb G}) &
\xrightarrow{\,\,\,\rho\,\,\,} & V & \longrightarrow & 0\\
&& \Big\Vert &&\,\,\, \Big\downarrow\varphi &&\,\,\, \Big\downarrow \phi\\
0 & \longrightarrow & {\rm ad}(E_{\mathbb G}) & \stackrel{\iota}{\longrightarrow} &
{\rm At}(E_{\mathbb G})& \xrightarrow{\,\,\,\varpi\,\,\,} & TX & \longrightarrow & 0
\end{matrix}
\end{equation}
where $\varphi$ and $\rho$ are constructed in \eqref{e10} and \eqref{e9} respectively.
Note that every vector bundle in \eqref{e11} is equipped with an action of $\Gamma$, and
all the homomorphisms in \eqref{e11} are $\Gamma$--equivariant.

\begin{definition}\label{def2}
An \textit{equivariant holomorphic Lie algebroid connection} on $E_{\mathbb G}$ is a $\Gamma$--equivariant
holomorphic homomorphism
$$
\delta\ : \ V\ \longrightarrow\ \mathcal{A}(E_{\mathbb G})
$$
such that $\rho\circ\delta\,=\, {\rm Id}_V$, where $\rho$ is the homomorphism in \eqref{e9}.
\end{definition}

Let $\delta$ be an equivariant holomorphic Lie algebroid connection on $E_{\mathbb G}$. For locally
defined holomorphic sections $s$ and $t$ of $V$, consider
$$
\alpha(s,\, t)\ :=\ [\delta(s),\, \delta(t)] - \delta([s,\, t]).
$$
For a locally defined holomorphic function $f$ on $X$,
$$
f\cdot \alpha(s,\, t)\,=\, \alpha(fs,\, t)\,=\, \alpha(s,\, ft)\,=\, -\alpha(ft,\, s).
$$
Also, $\rho(\alpha(s,\, t))\,=\, 0$, where $\rho$ is the homomorphism in \eqref{e9};
consequently, $\alpha(s,\, t)$ is a locally defined section of
$\text{ad}(E_{\mathbb G})$. From these it follows that $\alpha$ defines a
$\Gamma$--invariant holomorphic section
\begin{equation}\label{e12}
{\mathcal K}(\delta)\ \in\ H^0(X,\, \text{ad}(E_{\mathbb G})\otimes\bigwedge\nolimits^2 V^*)^\Gamma .
\end{equation}
The section ${\mathcal K}(\delta)$ in \eqref{e12} is the \textit{curvature} of the
equivariant holomorphic Lie algebroid connection $\delta$.

When $V\,=\, TX$ and $\phi\,=\, {\rm Id}_{TX}$, an equivariant holomorphic Lie algebroid connection on
$E_{\mathbb G}$ is a usual equivariant holomorphic connection on the principal
$\mathbb G$--bundle $E_{\mathbb G}$.

When ${\mathbb G}\,=\, \text{GL}(r,{\mathbb C})$, the notions of Lie algebroid connection and
curvature coincide with those for holomorphic vector bundles.

\section{Equivariant holomorphic connections and split Lie algebroids}\label{sec:Equivariant holomorphic 
connections and split Lie algebroids}

Earlier the notation $\mathbb G$ was used to denote a complex Lie group. Now-onwards, we will consider
principal bundles whose structure group is a connected reductive affine algebraic group defined
over $\mathbb C$. To distinguish it from a general complex Lie group, the notation $G$ will be used
instead of $\mathbb G$.
 
\subsection{Equivariant holomorphic connections}\label{subsec:Equivariant holomorphic connections}

Let $G$ be a connected reductive affine algebraic group defined over $\mathbb C$.
Take any parabolic subgroup $P\, \subset G$. Let $R_u(P)\, \subset\, P$ be the unipotent radical
of $P$. A \textit{Levi subgroup} of $P$ is a connected reductive complex algebraic subgroup
$L(P)\, \subset\, P$ such that the following composition of homomorphisms is an isomorphism:
$$
L(P)\ \hookrightarrow\ P \ \longrightarrow\ P/R_u(P)
$$
(see \cite[p.~125]{Hu}, \cite{Bo}).

Take a holomorphic principal $G$--bundle over $X$. Take a holomorphic character $\chi\, :\, L(P)\, 
\longrightarrow\, {\mathbb G}_m \,=\,{\mathbb C}^*$ of a Levi subgroup $L(P)$ of a parabolic subgroup $P$ of 
$G$. Let $E_{L(P)}\,\subset\, E_G$ be a holomorphic reduction of structure group of $E_G$ to $L(P)\, 
\subset\, G$. Let $E_{L(P)}\times^P {\mathbb C}^*$ be the holomorphic principal ${\mathbb C}^*$--bundle on 
$X$ obtained by extending the structure group of $E_{L(P)}$ using the character $\chi$. Using the standard 
multiplication action of ${\mathbb C}^*$ on $\mathbb C$, the principal ${\mathbb C}^*$--bundle 
$E_{L(P)}\times^P {\mathbb C}^*$ produces a holomorphic line bundle ${\mathcal L}(E_{L(P)},\chi) 
\,\longrightarrow\, X$.

The principal $G$--bundle $E_G$ admits a holomorphic connection if and only if for every triple
$(P,\, L(P),\, \chi)$ as above, and every holomorphic reduction of structure group $E_{L(P)}$ of
$E_G$ to $L(P)$, we have
$$
\text{degree}({\mathcal L}(E_{L(P)},\chi)) \ =\ 0
$$
\cite[Theorem 4.1]{AB}.

Let $E_G$ be an equivariant principal $G$--bundle over $X$. A reduction of structure group $E_{L(P)}$
of $E_G$ to $L(P)\, \subset\, G$ is called \textit{equivariant} if the action of $\Gamma$ on $E_G$ preserves the
submanifold $E_{L(P)}\, \subset\, E_G$. The following lemma gives a criterion for the existence of
an equivariant holomorphic connection on an equivariant principal $G$--bundle.

\begin{lemma}\label{lem1}
An equivariant principal $G$--bundle $E_G$ over $X$ admits an equivariant holomorphic connection if and
only if for every triple $(P,\, L(P),\, \chi)$ as above, and every $\Gamma$--equivariant holomorphic
reduction of structure group $E_{L(P)}$ of $E_G$ to $L(P)$,
\begin{equation}\label{e13}
{\rm degree}({\mathcal L}(E_{L(P)},\chi)) \ =\ 0.
\end{equation}
\end{lemma}

\begin{proof}
First, assume that $E_G$ admits an equivariant holomorphic connection. Then, from the above criterion of
\cite{AB} it follows immediately that \eqref{e13} holds.

To prove converse, assume that \eqref{e13} holds for every triple $(P,\, L(P),\, \chi)$ as above, and
every $\Gamma$--equivariant holomorphic reduction of structure group $E_{L(P)}$ of $E_G$ to $L(P)$.
This implies that the principal $G$--bundle $E_G$ admits a holomorphic connection (see
\cite[p.~274, Lemma 4.2]{Bi}). Let
$$
\delta\ : \ TX\ \longrightarrow\ \text{At}(E_G)
$$
be a holomorphic connection on $E_G$; so we have $\varpi\circ\delta\,=\, {\rm Id}_{TX}$, where $\varpi$
is the homomorphism in \eqref{e7} (see \cite{At}). For any $\gamma\, \in\, \Gamma$, let
$$
\delta_\gamma\ : \ TX\ \longrightarrow\ \text{At}(E_G)
$$
be the homomorphism given by the following composition of maps:
$$
TX \, \xrightarrow{\,\,\, \gamma\cdot\,\,\,}\ TX \ \xrightarrow{\,\,\,\delta\,\,\,}\
\text{At}(E_G) \ \xrightarrow{\,\,\, \gamma^{-1}\cdot\,\,\,}\ \text{At}(E_G),
$$
where $\gamma\cdot$ (respectively, $\gamma^{-1}\cdot$) is the action of $\gamma$ (respectively,
$\gamma^{-1}$) on $TX$ (respectively, $\text{At}(E_G)$); recall that $\Gamma$ acts on
both $TX$ and $\text{At}(E_G)$.

Now consider the homomorphism
$$
\widehat{\delta}\ :=\ \frac{1}{\# \Gamma} \sum_{\gamma\in\Gamma} \delta_\gamma\ :\
TX\ \longrightarrow\ \text{At}(E_G).
$$
Since $\varpi\circ\delta\,=\, {\rm Id}_{TX}$, it follows immediately that
$\varpi\circ\widehat{\delta}\,=\, {\rm Id}_{TX}$. It is also evident that $\widehat{\delta}$
is $\Gamma$--equivariant. Consequently, $\widehat{\delta}$ is
an equivariant holomorphic connection on the equivariant principal $G$--bundle $E_G$.
\end{proof}

The second part of the proof of Lemma \ref{lem1} gives the following:

\begin{corollary}\label{cor-2}
An equivariant principal $G$--bundle $E_G$ admits a holomorphic connection if
and only if $E_G$ admits an equivariant holomorphic connection.
\end{corollary}

\subsection{Split equivariant Lie algebroid connections}

Let $(V,\, \phi)$ be a split equivariant Lie algebroid (see Definition \ref{def1}).
As before, $G$ is a connected reductive affine algebraic group defined over $\mathbb C$.

\begin{proposition}\label{prop1}
An equivariant principal $G$--bundle $E_G$ over $X$ admits an equivariant holomorphic Lie algebroid 
connection (see Definition \ref{def2}) if and only if 
for every triple $(P,\, L(P),\, \chi)$ as in Lemma \ref{lem1}, and every $\Gamma$--equivariant holomorphic
reduction of structure group $E_{L(P)}$ of $E_G$ to $L(P)$,
\begin{equation}\label{e14}
{\rm degree}({\mathcal L}(E_{L(P)},\chi)) \ =\ 0.
\end{equation}
\end{proposition}

\begin{proof}
We will show that $E_G$ admits an equivariant holomorphic Lie algebroid connection if and only if $E_G$ admits an
equivariant holomorphic connection. To prove this, first assume that $E_G$ admits an
equivariant holomorphic connection. Take an equivariant holomorphic connection
$$
\delta_0\ :\ TX\ \longrightarrow\ \text{At}(E_G)
$$
on $E_G$. Since $\varpi\circ\delta_0\,=\, {\rm Id}_{TX}$, where $\varpi$ is the homomorphism in
\eqref{e7}, there is a unique holomorphic homomorphism
$$
\delta'_0\ :\ \text{At}(E_G)\ \longrightarrow\ \text{ad}(E_G)
$$
such that $\text{kernel}(\delta'_0)\,=\, \delta_0(TX)$ and $\delta'_0\circ\iota\,=\,
{\rm Id}_{\text{ad}(E_G)}$, where $\iota$ is the homomorphism in \eqref{e7}. Now, consider the
homomorphism
$$
\delta'_0\circ\varphi\ :\ \mathcal{A}(E_{\mathbb G})\ \longrightarrow\ \text{ad}(E_G),
$$
where $\varphi$ is the homomorphism in \eqref{e10}. There is a unique holomorphic homomorphism
$$
\delta\ :\ V\ \longrightarrow\ \mathcal{A}(E_{\mathbb G})
$$
such that $\delta(V)\,=\,\text{kernel}(\delta'_0\circ\varphi)$ and $\rho\circ\delta
\,=\, {\rm Id}_V$, where $\rho$ is the homomorphism in \eqref{e9}. Since $\delta$ is also
$\Gamma$--equivariant, it defines an equivariant holomorphic Lie algebroid connection on $E_G$.

To prove the converse, assume that $E_G$ has an equivariant holomorphic Lie algebroid connection
$$
\delta\ :\ V\ \longrightarrow\ \mathcal{A}(E_{\mathbb G}).
$$
Fix a $\Gamma$--equivariant holomorphic homomorphism
$$
\eta\ :\ TX \ \longrightarrow\ V
$$
such that $\phi\circ\eta\,=\, {\rm Id}_{TX}$; see Definition \ref{def1}
(recall that $(V,\, \phi)$ is a split equivariant Lie algebroid). Now it is straightforward
to check that the composition of homomorphisms
$$
\varphi\circ\delta\circ\eta\ :\ TX \ \longrightarrow\ \text{At}(E_G),
$$
where $\varphi$ is the homomorphism in \eqref{e10}, is an equivariant holomorphic connection
on $E_G$.

Since $E_G$ admits an equivariant holomorphic Lie algebroid connection if and only if $E_G$ admits an
equivariant holomorphic connection, Lemma \ref{lem1} completes the proof of the proposition.
\end{proof}

Proposition \ref{prop1}, Corollary \ref{cor-2} and Lemma \ref{lem1}
together give the following:

\begin{corollary}\label{cor-1}
Let $(V,\, \phi)$ be a split equivariant Lie algebroid and $G$ a reductive affine algebraic group over
$\mathbb C$. Let $E_G$ be an equivariant principal $G$--bundle over $X$. The following four
statements are equivalent:
\begin{enumerate}
\item $E_G$ over $X$ admits an equivariant holomorphic Lie algebroid connection.

\item $E_G$ admits an equivariant holomorphic connection.

\item $E_G$ admits a holomorphic connection.

\item For every triple $(P,\, L(P),\, \chi)$ as in Lemma \ref{lem1}, and every
$\Gamma$--equivariant holomorphic reduction of structure group $E_{L(P)}$ of $E_G$ to $L(P)$,
$$
{\rm degree}({\mathcal L}(E_{L(P)},\chi)) \ =\ 0.
$$
\end{enumerate}
\end{corollary}

\section{Nonsplit equivariant Lie algebroid connections}\label{sec:Nonsplit Lie algebroid connections}

Let $(V,\, \phi)$ be a nonsplit equivariant Lie algebroid (see Definition \ref{def1}). As before, $G$
is a connected reductive affine algebraic group defined over $\mathbb C$. We will show that
any equivariant principal $G$--bundle over $X$ admits an equivariant holomorphic Lie algebroid connection.

\begin{remark}\label{rem1}
Take an equivariant principal $G$--bundle $E_G$ over $X$. There is a Levi subgroup $L(P)$ of
a parabolic subgroup $P\, \subset\, G$, and a holomorphic reduction of structure group
$E_{L(P)}\, \subset\, E_G$ of $E_G$ to $L(P)\,\subset\, G$, satisfying the following conditions:
\begin{enumerate}
\item The action of $\Gamma$ on $E_G$ preserves $E_{L(P)}\, \subset\, E_G$, and

\item the maximal torus of the group of all $\Gamma$--equivariant holomorphic automorphisms of $E_{L(P)}$
coincides with the center of $L(P)$. (Note that any element $z$ of the center of $L(P)$ gives
a $\Gamma$--equivariant holomorphic automorphisms of $E_{L(P)}$ defined by $x\, \longmapsto\, xz$.)
\end{enumerate}
Moreover if $P'\, \subset\, G$ is another parabolic subgroup, $L(P')$ is a Levi subgroup of $P'$, and
$E_{L(P')}\, \subset\, E_G$ is a holomorphic reduction of structure group of $E_G$ to $L(P')$ satisfying
the above two conditions, then there is an element $x\, \in\, G$ such that $L(P')\,=\, x^{-1}L(P)x$
and $E_{L(P')}\,=\, E_{L(P)}x$. (See \cite[p.~63, Theorem 4.1]{BP}.)
\end{remark}

\begin{lemma}\label{lem2}
Assume that the equivariant principal $L(P)$--bundle $E_{L(P)}$ in Remark \ref{rem1} admits an equivariant
holomorphic Lie algebroid
connection. Then the equivariant principal $G$--bundle $E_G$ admits an equivariant
holomorphic Lie algebroid connection.
\end{lemma}

\begin{proof}
There are natural homomorphisms $a\, :\, \text{ad}(E_{L(P)})\, \hookrightarrow\, \text{ad}(E_G)$ and $b\,:\,
\text{At} (E_{L(P)})\, \hookrightarrow\, \text{At}(E_G)$ because $E_{L(P)}$ is a holomorphic
reduction of structure group of $E_G$ to $L(P)$, and they fit in the following commutative diagram:
$$
\begin{matrix}
0 &\longrightarrow & {\rm ad}(E_{L(P)}) & \longrightarrow & \text{At}(E_{L(P)}) &
\longrightarrow & TX & \longrightarrow & 0\\
&& \,\, \Big\downarrow a && \,\, \Big\downarrow b && \Big\Vert\\
0 & \longrightarrow & {\rm ad}(E_{G}) & \longrightarrow &
{\rm At}(E_G)& \longrightarrow & TX & \longrightarrow & 0
\end{matrix}
$$
where the rows are the Atiyah exact sequences (for $E_{L(P)}$ and $E_G$); see \eqref{e7}. This
commutative diagram produces the following commutative diagram:
\begin{equation}\label{e15}
\begin{matrix}
0 &\longrightarrow & {\rm ad}(E_{L(P)}) & \xrightarrow{\,\,\, \iota'\,\,\,} & \mathcal{A}(E_{L(P)}) &
\xrightarrow{\,\,\, \rho'\,\,\,} & V & \longrightarrow & 0\\
&& \,\, \Big\downarrow a' && \,\, \Big\downarrow b' && \Big\Vert\\
0 & \longrightarrow & {\rm ad}(E_{G}) & \xrightarrow{\,\,\, \iota\,\,\,} &
\mathcal{A}(E_G) & \longrightarrow & V & \longrightarrow & 0
\end{matrix}
\end{equation}
(see \eqref{e11}).

Since the principal $L(P)$--bundle $E_{L(P)}$ admits an equivariant
holomorphic Lie algebroid connection, we have a
$\Gamma$--equivariant holomorphic homomorphism $$\delta'\ :\ V\ \longrightarrow\ \mathcal{A}(E_{L(P)})$$
such that $\rho'\circ\delta'\,=\, \text{Id}_V$, where $\rho'$ is the homomorphism in \eqref{e15}. Now,
the homomorphism $$b'\circ\delta'\,:\, V\, \longrightarrow\, \mathcal{A}(E_G),$$ where $b'$
is the homomorphism in \eqref{e15}, is an equivariant holomorphic Lie algebroid connection on $E_G$.
\end{proof}

As before, $L(P)$ and $E_{L(P)}$ are as in Remark \ref{rem1}. Consider the Atiyah exact sequence
\begin{equation}\label{e16}
0 \ \longrightarrow\ {\rm ad}(E_{L(P)})\ \longrightarrow \ \text{At}(E_{L(P)})
\ \longrightarrow \ TX\ \longrightarrow\ 0
\end{equation}
for the equivariant principal ${L(P)}$--bundle $E_{L(P)}$. Let
$$
\beta\ \in\ H^1(X,\, \text{ad}(E_{L(P)})\otimes K_X)
$$
be the extension class for the short exact sequence in \eqref{e16}. Since \eqref{e16} is an exact
sequence of $\Gamma$--equivariant vector bundles, we have
\begin{equation}\label{e17}
\beta\ \in\ H^1(X,\, \text{ad}(E_{L(P)})\otimes K_X)^\Gamma \ 
\subset\ H^1(X,\, \text{ad}(E_{L(P)})\otimes K_X).
\end{equation}
Consider the dual homomorphism $\phi^*\, :\, K_X\, \longrightarrow\, V^*$ for the anchor map.
Tensoring it with the identity map of $\text{ad}(E_{L(P)})$, we have the homomorphism
$$
\Psi\ :=\ \text{Id}_{\text{ad}(E_{L(P)})}\otimes\phi^*\ :\ \text{ad}(E_{L(P)})\otimes K_X\
\longrightarrow\ \text{ad}(E_{L(P)})\otimes V^*.
$$
Let
\begin{equation}\label{e18}
\Psi_*\ :\ H^1(X,\, \text{ad}(E_{L(P)})\otimes K_X)\ \longrightarrow\
H^1(X,\, \text{ad}(E_{L(P)})\otimes V^*)
\end{equation}
be the homomorphism of cohomologies induced by the above homomorphism $\Psi$.

\begin{lemma}\label{lem3}
The equivariant principal $L(P)$--bundle $E_{L(P)}$ admits an equivariant holomorphic
Lie algebroid connection if and only if 
$$
\Psi_* (\beta)\ =\ 0,
$$
where $\beta$ and $\Psi_*$ are constructed in \eqref{e17} and \eqref{e18} respectively.
\end{lemma}

\begin{proof}
Consider the short exact sequence
\begin{equation}\label{e19}
0 \ \longrightarrow\ {\rm ad}(E_{L(P)})\ \xrightarrow{\,\,\, \iota'\,\,\,}\ \mathcal{A}(E_{L(P)})
\ \xrightarrow{\,\,\, \rho'\,\,\,}\ V\ \longrightarrow\ 0
\end{equation}
in \eqref{e15}. Let
$$
\beta_V\ \in\ H^1(X,\, \text{ad}(E_{L(P)})\otimes V^*)
$$
be the extension class for it. Since \eqref{e19} is an exact sequence of $\Gamma$--equivariant vector
bundles, we have
\begin{equation}\label{e20}
\beta_V\ \in\ H^1(X,\, \text{ad}(E_{L(P)})\otimes V^*)^\Gamma \ 
\subset\ H^1(X,\, \text{ad}(E_{L(P)})\otimes V^*).
\end{equation}
Note that the equivariant principal $L(P)$--bundle $E_{L(P)}$ admits an equivariant holomorphic Lie algebroid
connection if and only if we have $\beta_V\,=\, 0$; indeed, $\beta_V\,=\, 0$ if and only if
$E_{L(P)}$ admits a holomorphic Lie algebroid connection, and, exactly as shown in Corollary
\ref{cor-2}, $E_{L(P)}$ admits a holomorphic Lie algebroid connection if and only if if admits an
equivariant holomorphic Lie algebroid connection.

Now consider the commutative diagram
\begin{equation}\label{e21}
\begin{matrix}
0 &\longrightarrow & {\rm ad}(E_{L(P)}) & \xrightarrow{\,\,\,\iota'\,\,\,} & {\mathcal A}(E_{L(P)}) &
\xrightarrow{\,\,\,\rho'\,\,\,} & V & \longrightarrow & 0\\
&& \Big\Vert && \Big\downarrow && \,\, \Big\downarrow\phi \\
0 & \longrightarrow & {\rm ad}(E_{L(P)}) & \longrightarrow &
{\rm At}(E_{L(P)})& \longrightarrow & TX & \longrightarrow & 0
\end{matrix}
\end{equation}
(see \eqref{e11}). From \eqref{e21} it follows immediately that
\begin{equation}\label{e22}
\Psi_* (\beta)\ =\ \beta_V,
\end{equation}
where $\beta$ and $\beta_V$ are the extension classes in \eqref{e17} and \eqref{e20} respectively
while $\Psi_*$ is the homomorphism in \eqref{e18}. Since $E_{L(P)}$ admits an equivariant holomorphic
Lie algebroid connection if and only if we have $\beta_V\,=\, 0$, it follows from \eqref{e22} that
$E_{L(P)}$ admits an equivariant holomorphic Lie algebroid connection if and only if $\Psi_* (\beta)\,=\, 0$.
\end{proof}

As before, $L(P)$ and $E_{L(P)}$ are as in Remark \ref{rem1}.
Denote the Lie algebra of $L(P)$ by $\ell({\mathfrak p})$. Since $L(P)$ is reductive, there is
an $L(P)$--invariant nondegenerate symmetric bilinear form on $\ell({\mathfrak p})$. Fix a 
$L(P)$--invariant nondegenerate symmetric bilinear form
\begin{equation}\label{e23}
{\mathcal B}\ \in\ \text{Sym}^2(\ell({\mathfrak p})^*)^{L(P)}.
\end{equation}
The form $\mathcal B$ in \eqref{e23} produces a holomorphic isomorphism
\begin{equation}\label{e24}
\text{ad}(E_{L(P)}) \ \stackrel{\sim}{\longrightarrow}\ \text{ad}(E_{L(P)})^*.
\end{equation}

By Serre duality,
\begin{equation}\label{e25}
H^1(X,\, \text{ad}(E_{L(P)})\otimes K_X)^\Gamma \ = \ (H^0(X,\, \text{ad}(E_{L(P)})^*)^*)^\Gamma
\end{equation}
$$
 =\ (H^0(X,\, \text{ad}(E_{L(P)}))^*)^\Gamma \ =\ (H^0(X,\, \text{ad}(E_{L(P)}))^\Gamma)^* ;
$$
see \eqref{e24}.

Let
\begin{equation}\label{e26}
\widehat{\beta}\ \in\ (H^0(X,\, \text{ad}(E_{L(P)}))^\Gamma)^*\ =\ 
\text{Hom}(H^0(X,\, \text{ad}(E_{L(P)}))^\Gamma,\, {\mathbb C})
\end{equation}
be the element corresponding to $\beta$ in \eqref{e17} for the isomorphism in \eqref{e25}.

Let
\begin{equation}\label{e27}
{\mathcal Z}(\ell({\mathfrak p}))\ \subset \ \ell({\mathfrak p})
\end{equation}
be the center of $\ell({\mathfrak p})$. Since the adjoint action of $L(P)$ on its Lie algebra
$\ell({\mathfrak p})$ fixes ${\mathcal Z}(\ell({\mathfrak p}))$ pointwise, we have an injective homomorphism
\begin{equation}\label{e28}
\Phi\ :\ {\mathcal Z}(\ell({\mathfrak p}))\ \longrightarrow\ H^0(X,\, \text{ad}(E_{L(P)}))^\Gamma .
\end{equation}

\begin{proposition}\label{prop2}
\mbox{}
\begin{enumerate}
\item Take any $\xi_n\, \in\, H^0(X,\, {\rm ad}(E_{L(P)}))^\Gamma$ which is nilpotent over some point
of $X$. Then $$\widehat{\beta}(\xi_n)\ =\ 0,$$ where $\widehat{\beta}$ is the homomorphism in \eqref{e26}.

\item Take any $\xi_s\, \in\, H^0(X,\, {\rm ad}(E_{L(P)}))^\Gamma$ which is semisimple over every point
of $X$. Then there is an element $w\, \in\, {\mathcal Z}(\ell({\mathfrak p}))$ such that
$$
\Phi(w)\ =\ \xi_s,
$$
where $\Phi$ is the homomorphism in \eqref{e28}.
\end{enumerate}
\end{proposition}

\begin{proof}
The first statement follows immediately from \cite[p.~341, Proposition 3.9]{AB}.

For the proof of second statement, first recall from Remark \ref{rem1} that the center of $L(P)$ is the
maximal torus of the group of all $\Gamma$--equivariant holomorphic automorphisms of $E_{L(P)}$.
The automorphisms of $E_{L(P)}$ given by the center of $L(P)$ evidently commute with all the automorphisms
of $E_{L(P)}$. On the other hand, the image of the connected component, containing the
identity element, of the center of $L(P)$ under the natural map to the 
group of all $\Gamma$--equivariant holomorphic automorphisms of $E_{L(P)}$, is a maximal torus
of the group of all $\Gamma$--equivariant holomorphic automorphisms of $E_{L(P)}$. Therefore, we
conclude that a maximal torus of the 
group of all $\Gamma$--equivariant holomorphic automorphisms of $E_{L(P)}$
is contained in the center of the group of all $\Gamma$--equivariant
holomorphic automorphisms of $E_{L(P)}$.

If the maximal torus of a connected complex reductive algebraic group $\mathcal G$
is contained in the center of $\mathcal G$, then $\mathcal G$ is abelian, which
means that $\mathcal G$ is a torus. Therefore, the semisimple
part, i.e., the Levi factor, of the Lie algebra of the group of all $\Gamma$--equivariant holomorphic
automorphisms of $E_{L(P)}$ coincides with the center ${\mathcal Z}(\ell({\mathfrak p}))$. Note that
$H^0(X,\, \text{ad}(E_{L(P)}))^\Gamma$ is the Lie algebra of the group of all $\Gamma$--equivariant
holomorphic automorphisms of $E_{L(P)}$. From these, the second statement of the proposition
follows immediately.
\end{proof}

\begin{proposition}\label{prop3}
The equivariant principal $L(P)$--bundle $E_{L(P)}$ admits an equivariant holomorphic
Lie algebroid connection.
\end{proposition}

\begin{proof}
In view of \eqref{e22} and Lemma \ref{lem3}, it suffices to show that
\begin{equation}\label{e29}
\beta_V\ =\ \Psi_* (\beta)\ =\ 0,
\end{equation}
where $\beta$ and $\Psi_*$ are constructed in \eqref{e17} and \eqref{e18} respectively.

By Serre duality,
\begin{equation}\label{e30}
H^1(X,\, \text{ad}(E_{L(P)})\otimes V^*)^\Gamma \ = \ (H^0(X,\, \text{ad}(E_{L(P)})^*\otimes
V\otimes K_X)^*)^\Gamma
\end{equation}
$$
 =\ (H^0(X,\, \text{ad}(E_{L(P)})\otimes V\otimes K_X)^*)^\Gamma \ =\
(H^0(X,\, \text{ad}(E_{L(P)})\otimes V\otimes K_X)^\Gamma)^* ;
$$
see \eqref{e24}. Let
\begin{equation}\label{e31}
\widehat{\beta}_V\ \in\ (H^0(X,\, \text{ad}(E_{L(P)})\otimes V\otimes K_X)^\Gamma)^*\ =\ 
\text{Hom}(H^0(X,\, \text{ad}(E_{L(P)})\otimes V\otimes K_X)^\Gamma,\, {\mathbb C})
\end{equation}
be the element corresponding to $\beta_V$ in \eqref{e20} for the isomorphism in \eqref{e30}.

Consider the anchor map $\phi\, \in\, H^0(X,\, V^*\otimes TX)^\Gamma$. We have the homomorphism
\begin{equation}\label{e32}
\Phi_1\ :\ H^0(X,\, \text{ad}(E_{L(P)})\otimes V\otimes K_X)^\Gamma\ \longrightarrow\
H^0(X,\, \text{ad}(E_{L(P)})\otimes V\otimes K_X)^\Gamma\otimes H^0(X,\, V^*\otimes TX)^\Gamma
\end{equation}
that sends any $s\, \in\, H^0(X,\, \text{ad}(E_{L(P)})\otimes V\otimes K_X)^\Gamma$ to
$s\otimes\phi$. There is a natural homomorphism
\begin{equation}\label{e33}
\Phi_2\ :\ H^0(X,\, \text{ad}(E_{L(P)})\otimes V\otimes K_X)^\Gamma
\otimes H^0(X,\, V^*\otimes TX)^\Gamma\ \longrightarrow\
\end{equation}
$$
H^0(X,\, \text{ad}(E_{L(P)})\otimes V\otimes K_X\otimes V^*\otimes TX)^\Gamma
\ =\ H^0(X,\, \text{ad}(E_{L(P)})\otimes \text{End}(V)\otimes\text{End}(TX))^\Gamma. 
$$
Using the trace maps
\begin{equation}\label{etr}
\text{End}(V)\, \longrightarrow\, {\mathcal O}_X\ \ \text{ and }\ \ \text{End}(TX)\,
\longrightarrow\, {\mathcal O}_X,
\end{equation}
we have the map
\begin{equation}\label{e34}
\Phi_3\ :\ H^0(X,\, \text{ad}(E_{L(P)})\otimes \text{End}(V)\otimes\text{End}(TX))^\Gamma\
\longrightarrow\ H^0(X,\, \text{ad}(E_{L(P)}))^\Gamma.
\end{equation}
Now consider the homomorphism
\begin{equation}\label{e34b}
\widetilde{\Phi}\ :=\ \Phi_3\circ\Phi_2\circ\Phi_1\ :\ H^0(X,\, \text{ad}(E_{L(P)})\otimes V\otimes K_X)^\Gamma
\ \longrightarrow\ H^0(X,\, \text{ad}(E_{L(P)}))^\Gamma,
\end{equation}
where $\Phi_1$, $\Phi_2$ and $\Phi_3$ are constructed in \eqref{e32}, \eqref{e33} and \eqref{e34} respectively.
{}From \eqref{e22} we know that the following diagram is commutative:
$$
\begin{matrix}
H^0(X,\, \text{ad}(E_{L(P)})\otimes V\otimes K_X)^\Gamma & \xrightarrow{\,\,\, \widetilde{\Phi}\,\,\,} &
H^0(X,\, \text{ad}(E_{L(P)}))^\Gamma\\
\,\,\,\, \Big\downarrow\widehat{\beta}_V && \,\,\Big\downarrow\widehat{\beta}\\
{\mathbb C} & \xrightarrow{\,\,\, \text{Id}\,\,\,} & {\mathbb C}
\end{matrix}
$$
where $\widehat{\beta}$ and $\widehat{\beta}_V$ are the homomorphisms constructed in
\eqref{e26} and \eqref{e31} respectively, and $\widetilde{\Phi}$ is defined in \eqref{e34b}.
In other words, we have
\begin{equation}\label{e35}
\widehat{\beta}\circ \widetilde{\Phi}\ = \ \widehat{\beta}_V
\end{equation}
as elements of $\text{Hom}(H^0(X,\, \text{ad}(E_{L(P)})\otimes V\otimes K_X)^\Gamma, \,{\mathbb C})$.

To prove \eqref{e29} by contradiction, assume that
\begin{equation}\label{e36}
\beta_V\ =\ \Psi_* (\beta)\ \not=\ 0.
\end{equation}
{}From \eqref{e36} it follows that there is a section $s\, \in\, H^0(X,\, \text{ad}(E_{L(P)})\otimes
V\otimes K_X)^\Gamma$ such that
\begin{equation}\label{e37}
\widehat{\beta}_V (s) \ \not=\ 0,
\end{equation}
where $\widehat{\beta}_V$ is constructed in \eqref{e31}. Consider the section
\begin{equation}\label{e38}
\widehat{s} \ :=\ \widetilde{\Phi} (s) \ \in\ H^0(X,\, \text{ad}(E_{L(P)}))^\Gamma,
\end{equation}
where $s$ is the section in \eqref{e37} and $\widetilde{\Phi}$ is constructed in
\eqref{e34b}. From \eqref{e36}, \eqref{e35} and \eqref{e38}
we know that
\begin{equation}\label{e39}
\widehat{\beta}(\widehat{s})\ \not=\ 0.
\end{equation}

In view of \eqref{e39}, from Proposition \ref{prop2}(1) we know that for each point $x\, \in\, X$, the element
$\widehat{s}(x)\, \in\, \text{ad}(E_{L(P)})_x$ is \textit{not} nilpotent.

So for each point $x\in\, X$, the semisimple component of $\widehat{s}(x)\, \in\, \text{ad}(E_{L(P)})_x$,
for the Jordan decomposition, is nonzero. Moreover, the conjugacy class of the semisimple
component of $\widehat{s}(x)$ is actually independent of the point $x\,\in\, X$. To see this, take any
$L(P)$--invariant holomorphic function $I$ on $\ell({\mathfrak p})$. Then $x\, \longmapsto\, I(\widehat{s}(x))$
is a holomorphic function on $X$. This function is a constant one because $X$ is compact and connected. From
this it follows that the conjugacy class of the semisimple
component of $\widehat{s}(x)$ is independent of $x\,\in\, X$.

Take a holomorphic character $\chi\, :\, L(P)\, \longrightarrow\, {\mathbb G}_m\,=\, {\mathbb C}^*$.
Let $$d\chi\ :\ \ell({\mathfrak p})\ \longrightarrow\ {\mathbb C}$$ be the homomorphism of Lie algebras
given by $\chi$. This homomorphism $d\chi$ produces a homomorphism
\begin{equation}\label{e40a}
\widetilde{\chi}\ :\ \text{ad}(E_{L(P)})\ \longrightarrow\ {\mathcal O}_X.
\end{equation}
Let
\begin{equation}\label{e40}
\widetilde{\chi}_*\ :\ H^0(X,\, \text{ad}(E_{L(P)}))^\Gamma\ \longrightarrow\ H^0(X,\, {\mathcal O}_X)\
= \ {\mathbb C}
\end{equation}
be the homomorphism of global sections given by $\widetilde{\chi}$ in \eqref{e40a}.

{}From Proposition \ref{prop2}(2) it follows that there is a holomorphic character 
$\chi\, :\, L(P)\, \longrightarrow\, {\mathbb C}^*$ such that
\begin{equation}\label{e41}
\widetilde{\chi}_*(\widehat{s})\ \not=\ 0,
\end{equation}
where $\widehat{s}$ and $\widetilde{\chi}_*$ are constructed in \eqref{e38} and \eqref{e40} respectively.

The homomorphism $\widetilde{\chi}$ in \eqref{e40a} produces a homomorphism
\begin{equation}\label{e42}
\widetilde{\chi}'\ :\ H^0(X,\, \text{ad}(E_{L(P)})\otimes V\otimes K_X)^\Gamma
\ \longrightarrow\ H^0(X,\, V\otimes K_X)^\Gamma .
\end{equation}
Define the map
\begin{equation}\label{e43}
\Psi_1\ :\ H^0(X,\, V\otimes K_X)^\Gamma \ \longrightarrow\ H^0(X,\, V\otimes K_X)\otimes H^0(X,\,
V^*\otimes TX)^\Gamma, \ \ \, v\ \longmapsto\ v\otimes\phi,
\end{equation}
where $\phi$ is the anchor map. We have the natural map
\begin{equation}\label{e44}
\Psi_2\ :\ H^0(X,\, V\otimes K_X)^\Gamma\otimes H^0(X,\, V^*\otimes TX)^\Gamma\ \longrightarrow\
H^0(X,\, V\otimes K_X \otimes V^*\otimes TX)^\Gamma
\end{equation}
$$
 = \ H^0(X,\, \text{End}(V)\otimes \text{End}(TX))^\Gamma .
$$
Using the trace maps in \eqref{etr} we have the homomorphism
\begin{equation}\label{e45}
\Psi_3\ :\ H^0(X,\, \text{End}(V)\otimes \text{End}(TX))^\Gamma \ \longrightarrow\ H^0(X,\, {\mathcal O}_X)\
=\ {\mathbb C}.
\end{equation}
Now define
\begin{equation}\label{e45b}
\widetilde{\Psi}\ :=\ \Psi_3 \circ\Psi_2\circ\Psi_1\ :\ H^0(X,\, V\otimes K_X) \ \longrightarrow\ 
{\mathbb C},
\end{equation}
where $\Psi_1$, $\Psi_2$ and $\Psi_3$ are constructed in \eqref{e43}, \eqref{e44} and \eqref{e45}
respectively. We note that the following diagram is commutative:
$$
\begin{matrix}
H^0(X,\, \text{ad}(E_{L(P)})\otimes V\otimes K_X)^\Gamma
& \xrightarrow{\,\,\, \widetilde{\chi}'\,\,\,}& H^0(X,\, V\otimes K_X)^\Gamma\\
\,\,\,\,\Big\downarrow \widetilde{\Phi} && \,\,\,\, \Big\downarrow\widetilde{\Psi}\\
{\mathbb C} & \xrightarrow{\,\,\, \text{Id}\,\,\,} & {\mathbb C}
\end{matrix}
$$
where $\widetilde{\chi}'$, $\widetilde{\Phi}$ and $\widetilde{\Psi}$ are constructed in \eqref{e42},
\eqref{e34b} and \eqref{e45b} respectively. Consequently, using \eqref{e38} we have
\begin{equation}\label{e46}
\widetilde{\Psi}\circ\widetilde{\chi}'(s)\ = \ \widetilde{\chi}_*(\widehat{s})
\end{equation}
as elements of $\mathbb C$. From \eqref{e41} and \eqref{e46} we conclude that 
\begin{equation}\label{e47}
\widetilde{\chi}'(s)\ \not= \ 0.
\end{equation}
In view of the construction of $\widetilde{\chi}'$ done in \eqref{e42}, from \eqref{e47} it
is deduced that the following composition of maps
$$
TX \ \xrightarrow{\,\,\, \widetilde{\chi}'(s)\,\,\,}\ V \ \xrightarrow{\,\,\,\phi\,\,\,}\ TX
$$
coincides with multiplication by the nonzero constant $\widetilde{\chi}_*(\widehat{s})\,
\in\, {\mathbb C}\setminus\{0\}$ in \eqref{e41}.
Consequently, the homomorphism
$$
\frac{1}{\widetilde{\chi}_*(\widehat{s})}\cdot \widetilde{\chi}'(s)\ :\ TX \ \longrightarrow\ V
$$
gives a splitting of the equivariant Lie algebroid $(V,\, \phi)$. But $(V,\, \phi)$ does not split.
In view of this contradiction, it follows that \eqref{e36} does not hold. This completes the proof.
\end{proof}

\begin{corollary}\label{cor1}
The equivariant principal $G$--bundle $E_G$ admits an equivariant holomorphic
Lie algebroid connection.
\end{corollary}

\begin{proof}
This follows from the combination of Lemma \ref{lem2} and Proposition \ref{prop3}.
\end{proof}

\section{Criterion for Lie algebroid connection}\label{sec:Parabolic Lie algebroids and parabolic connections}

As before, $G$ is a complex reductive affine algebraic group. The combination of
Corollary \ref{cor-1} and Corollary \ref{cor1} gives the following:

\begin{theorem}\label{thm1}
\mbox{}
\begin{itemize}
\item Let $(V,\,\phi)$ be a nonsplit equivariant Lie algebroid. Then any equivariant
principal $G$--bundle over $X$ admits an equivariant holomorphic Lie algebroid connection.

\item Let $(V,\, \phi)$ be a split equivariant Lie algebroid.
Let $E_G$ be an equivariant principal $G$--bundle over $X$. The following four statements
are equivalent:
\begin{enumerate}
\item $E_G$ over $X$ admits an equivariant holomorphic Lie algebroid connection.

\item $E_G$ admits an equivariant holomorphic connection.

\item $E_G$ admits a holomorphic connection.

\item For every triple $(P,\, L(P),\, \chi)$ as in Lemma \ref{lem1}, and every
$\Gamma$--equivariant holomorphic reduction of structure group $E_{L(P)}$ of $E_G$ to $L(P)$,
$$
{\rm degree}({\mathcal L}(E_{L(P)},\chi)) \ =\ 0.
$$
\end{enumerate}
\end{itemize}
\end{theorem}

\begin{remark}\label{rem2}
Take a complex Lie group $\mathbb G$ and a holomorphic principal $\mathbb G$--bundle $E_{\mathbb G}$
over $X$. Consider the corresponding Lie algebroid $(\text{At}(E_{\mathbb G}),\, \varpi)$ as in \eqref{e7}.
Assume that $E_{\mathbb G}$ does not admit any holomorphic connection. It was shown in Example
\ref{ex1} that the Lie algebroid $(\text{At}(E_{\mathbb G}),\, \varpi)$ is nonsplit. Assume that
$E_{\mathbb G}$ is equivariant. Therefore, Theorem \ref{thm1} says that any equivariant principal $G$--bundle
$E_G$ over $X$ admits an equivariant holomorphic Lie algebroid connection for the Lie algebroid
$(\text{At}(E_{\mathbb G}),\, \varpi)$.

Next consider the case where the equivariant holomorphic principal $\mathbb G$--bundle $E_{\mathbb G}$
does not admit any holomorphic connection. Now Theorem \ref{thm1} says that the following four statements
are equivalent:
\begin{enumerate}
\item $E_G$ over $X$ admits an equivariant holomorphic Lie algebroid connection.

\item $E_G$ admits an equivariant holomorphic connection.

\item $E_G$ admits a holomorphic connection.

\item For every triple $(P,\, L(P),\, \chi)$ as in Lemma \ref{lem1}, and every
$\Gamma$--equivariant holomorphic reduction of structure group $E_{L(P)}$ of $E_G$ to $L(P)$,
$$
{\rm degree}({\mathcal L}(E_{L(P)},\chi)) \ =\ 0.
$$
\end{enumerate}
\end{remark}

We will reformulate Theorem \ref{thm1} in the set-up of parabolic bundles.

Fix $n$ ordered distinct points
\begin{equation}\label{g1}
S \ =\ \{s_1,\, \cdots,\, s_n\}\ \subset\ X.
\end{equation}
For each $1\, \leq\, i\, \leq n$, fix an integer $N_i\, \geq\, 2$. We assume the following:
\begin{enumerate}
\item If ${\rm genus}(X)\,=\, 0$, then $n\, \not=\,1$.

\item If ${\rm genus}(X)\,=\, 0$, and $n\, =\,2$, then $N_1\,=\, N_2$.
\end{enumerate}

A parabolic $G$--bundle consists of a complex manifold ${\mathcal E}_G$, a surjective holomorphic map
$p\, :\, {\mathcal E}_G\, \longrightarrow\, X$ and a holomorphic right action of $G$
$$
\Psi\ :\ {\mathcal E}_G\times G\ \longrightarrow\ {\mathcal E}_G
$$
such that the following conditions hold:
\begin{enumerate}
\item $\Psi(y,\, h)\,=\, \Psi(x)$ for all $x\, \in\, {\mathcal E}_G$ and $h\, \in\, G$.

\item For every $x\, \in\, X$, the action of $G$ on $p^{-1}(x)$ is transitive.

\item The restriction $p\big\vert_{p^{-1}(X\setminus S)}\, :\, p^{-1}(X\setminus S)\, \longrightarrow\,
X\setminus S$ (see \eqref{g1}) is a holomorphic principal $G$--bundle on $X\setminus S$.

\item The isotropy subgroup for any $y\, \in\, p^{-1}(s_i)$ is a finite cyclic subgroup of $G$ whose order
divides $N_i$.
\end{enumerate}
(See \cite{BBN}, \cite{Bi}, \cite{BS}.)
For $1\, \leq\, i\, \leq\, n$, the parabolic weights, at $s_i$, of a parabolic vector bundle will be
integral multiples of $\frac{1}{N_i}$.

The parabolic tangent bundle, denoted by $(TX)_*$, is $TX\otimes {\mathcal O}_X(-\sum_{i=1}^n s_i)$
equipped with parabolic weight $\frac{1}{N_i}$ at $s_i$, $1\,\leq\,i\, \leq\, n$.

A parabolic Lie algebroid is a pair $(V_*,\, \phi)$, where $V_*$ is a parabolic vector bundle, and
$\phi\, :\, V_*\, \longrightarrow\, (TX)_*$ is a parabolic homomorphism, such that
\begin{enumerate}
\item $V_*$ is equipped with a $\mathbb C$--Lie algebra structure
$$
[-,\, -] \,\,:\,\, V_*\otimes_{\mathbb C} V_* \,\, \longrightarrow\,\, V_*
$$
which is compatible with the parabolic structure, and

\item $[s,\, f\cdot t]\,=\, f\cdot [s,\, t]+\phi(s)(f)\cdot t$
for locally defined holomorphic sections
$s,\, t$ of $V_*$ and all locally defined holomorphic functions $f$ on $X$.
\end{enumerate}

A parabolic Lie algebra $(V_*,\, \phi)$ is called split if there is a parabolic homomorphism
$\beta\,:\, (TX)_*\, \longrightarrow\, V_*$ such that $\phi\circ\beta\, =\, {\rm Id}_{(TX)_*}$.
A parabolic Lie algebra $(V_*,\, \phi)$ is called nonsplit if it is not split.

Take a parabolic $G$--bundle $p\, :\, {\mathcal E}_G\, \longrightarrow\, X$. The invariant direct image,
on $X$, of the holomorphic tangent bundle $T{\mathcal E}_G$
$$
(p_*T{\mathcal E}_G)^G \ \subset\ p_*T{\mathcal E}_G
$$
has a natural parabolic structure. The resulting parabolic vector bundle is called the
Atiyah bundle for ${\mathcal E}_G$, and it is denoted by $\text{At}({\mathcal E}_G)_*$. Let $T_p\,\subset\,
T{\mathcal E}_G$ be the relative tangent bundle for the projection $p$. The parabolic adjoint bundle
$\text{ad}({\mathcal E}_G)_*$ is defined to be the invariant direct image.$(p_*T_p)^G$.
The parabolic vector bundle $\text{At}({\mathcal E}_G)_*$ fits in the following short exact sequence
of parabolic vector bundles over $X$:
\begin{equation}\label{f1}
0\ \longrightarrow\ \text{ad}({\mathcal E}_G)_* \ \longrightarrow\ \text{At}({\mathcal E}_G)_*
\ \stackrel{\psi}{\longrightarrow}\ (TX)_* \ \longrightarrow\ 0.
\end{equation}
A holomorphic connection on ${\mathcal E}_G$ is a holomorphic splitting of \eqref{f1} (see \cite{Bi}).

Define ${\mathcal A}({\mathcal E}_{\mathbb G})_*$ to be the parabolic vector bundle given by the kernel of
the parabolic homomorphism
$$
V^*\oplus \text{At}({\mathcal E}_G)_*\ \longrightarrow\ (TX)_*, \ \ \, (a,\, b)\ \longmapsto\
\phi(a)-\psi(b),
$$
where $\psi$ is the homomorphism in \eqref{f1}. The parabolic vector bundle $\mathcal{A}({\mathcal E}_G)_*$
fits in the following short exact sequence of parabolic vector bundles over $X$:
\begin{equation}\label{f2}
0\ \longrightarrow\ \text{ad}({\mathcal E}_G)_* \ \longrightarrow\ \mathcal{A}({\mathcal E}_G)_*
\ \longrightarrow\ V \ \longrightarrow\ 0.
\end{equation}
A holomorphic Lie algebroid connection on ${\mathcal E}_G$ is a holomorphic splitting of
\eqref{f2}.

There is a ramified Galois covering $\varpi\, :\, Y\, \longrightarrow\, X$ such that
\begin{enumerate}
\item the branch locus of $\varpi$ is $S \,=\, \{s_1,\, \cdots,\, s_n\}\, \subset\, X$ (see \eqref{g1}), and

\item for every $1\, \leq\, i\, \leq\, n$, the multiplicity of $\varpi$ at any $y\,\in\,
\varpi^{-1}(s_i)$ is $N_i$.
\end{enumerate}
(See \cite[p. 26, Proposition 1.2.12]{Na} for the
existence of such a covering $\varpi$.)

Let $\Gamma\,=\, \text{Gal}(\varpi)$ be the Galois group for $\varpi$.

The parabolic $G$--bundles over $X$ correspond to $\Gamma$--equivariant principal $G$--bundles on
$Y$ \cite{BBN}, \cite{BS}. In particular, the parabolic vector bundles over $X$ correspond to $\Gamma$--equivariant
vector bundles on $Y$. The parabolic Lie algebroids over $X$ correspond to the $\Gamma$--equivariant Lie
algebroids on $Y$. Parabolic $G$--bundles over $X$ equipped with a parabolic connection correspond to
the $\Gamma$--equivariant principal $G$--bundles on $Y$ equipped with a $\Gamma$--equivariant connection.

Consequently, Theorem \ref{thm1} gives the following:

\begin{theorem}\label{thm2}
\mbox{}
\begin{itemize}
\item Let $(V_*,\,\phi)$ be a nonsplit parabolic Lie algebroid. Then any
parabolic $G$--bundle over $X$ admits a parabolic Lie algebroid connection.

\item Let $(V_*,\, \phi)$ be a split parabolic Lie algebroid.
Let ${\mathcal E}_G$ be a parabolic $G$--bundle over $X$. The following three statements
are equivalent:
\begin{enumerate}
\item ${\mathcal E}_G$ over $X$ admits a parabolic Lie algebroid connection.

\item ${\mathcal E}_G$ admits a parabolic holomorphic connection.

\item For every triple $(P,\, L(P),\, \chi)$ as in Lemma \ref{lem1}, and every
holomorphic reduction of structure group ${\mathcal E}_{L(P)}$ of ${\mathcal E}_G$ to $L(P)$, the
parabolic line bundle over $X$ associated to ${\mathcal E}_{L(P)}$ for $\chi$ has parabolic degree zero.
\end{enumerate}
\end{itemize}
\end{theorem}

\section*{Acknowledgements}

We thank the referee for going through the paper very carefully and making numerous suggestions. The
author is partially supported by a J. C. Bose Fellowship (JBR/2023/000003).

%%%%%%%%%%%%%%%%%%%%%%%%%%%%%%%%%%%%%%%%%%%%%%%%%%%%%%%%%%%%%%%%%%%%%%%%%%%%%%%%%%%%%%%%%%%%%%%%%%

\end{document}